\documentclass[amssymb, 11pt]{amsart}
\usepackage{latexsym}

\newcommand{\F}{\mathbb{F}}
\newcommand{\FF}{\mathbb{F}}

\newcommand{\ind}{\operatorname{ind}}
\newcommand{\orb}{\operatorname{orb}}

\def\skipa{\vspace{-1.5mm} & \vspace{-1.5mm} & \vspace{-1.5mm}\\}

\newlength{\standardunitlength}
\newtheorem{prop}{Proposition}[section]

\newtheorem{lemma}[prop]{Lemma}
\newtheorem{cor}[prop]{Corollary}
\newtheorem{theorem}[prop]{Theorem}
\newtheorem{thm}[prop]{Theorem}
\newtheorem{example}[prop]{Example}

\begin{document}

\title [Derangements in finite classical groups] {Derangements in finite classical groups for actions related to extension field and imprimitive subgroups
and the solution of the Boston--Shalev conjecture}

\author{Jason Fulman}
\address{Department of Mathematics\\
        University of Southern California\\
        Los Angeles, CA, 90089-2532, USA}
\email{fulman@usc.edu}

\author{Robert Guralnick}
\address{Department of Mathematics\\University of Southern California \\ Los Angeles, CA, 90089-2532, USA}
\email{guralnic@usc.edu}

\date{Version of July 31, 2015}



\thanks{Fulman was partially supported by NSA grant H98230-13-1-0219.
Guralnick was partially supported by NSF grant DMS-1302886.}

\begin{abstract}  This is the fourth paper in a series. We prove a conjecture made independently by Boston et al and Shalev.
The conjecture asserts that there is an absolute positive constant $\delta$
such that if $G$ is a finite simple group acting transitively on a set of size $n > 1$, then the proportion of derangements in $G$ is greater
than $\delta$.  We show that with possibly finitely many exceptions, one can take $\delta = .016$.   Indeed, we prove much stronger results
showing that for many actions, the proportion of derangements goes to $1$ as $n$ increases and prove similar results for families of permutation
representations.
\end{abstract}

\maketitle

\section{Introduction}

A permutation on a set $X$ is called a derangement if it has no fixed points.  A classical and elementary theorem of Jordan asserts that a finite
group acting transitively on a set $X$ of size at least $2$ contains derangements.  There are many results on the proportion of derangements
in finite transitive groups.  Rather amazingly, it was only recently that it was shown \cite{CC} that if $G$ acts transitively on a set of size $n > 1$,
then the proportion of derangements is at least $1/n$  (this is a quite easy theorem -- see \cite{DFG} for a short proof and also an upper bound
in terms of the rank of the permutation group).

Derangements come up naturally in many contexts (see the surveys \cite{DFG} and \cite{Se} for applications to topology, number theory, and maps
between varieties over finite fields). Perhaps the earliest results on derangements are due to Montmort \cite{Mo}. He studied derangements
in the full symmetric group $S_n$ to analyze a card game (it is easy to see that the proportion of derangements in $S_n$ tends to $1/e$
as $n \rightarrow \infty$ and is always at least $1/3$).

If $G$ is a finite simple group acting faithfully and transitively on a set $X$, then it was noticed that the proportion of derangements never seemed
to be too small.  This led Boston et al \cite{BDF} and Shalev to (independently) conjecture that there is a constant $\delta > 0$ so that for
a finite simple group $G$, the proportion of derangements is at least $\delta$. (Boston et al \cite{BDF} suggest that $\delta = 2/7$; in fact that
is not true --  Tim Burness has observed that  the group ${^2}F_4(2)'$ has a transitive permutation representation with the proportion of derangements
equal to $89/325$. We also mention that by \cite{NP} that for $d$ sufficiently large, the proportion of derangements for $SL(d,2)$ acting on $1$-dimensional spaces is
less than $.29$).  In this paper, we complete the proof of the Boston-Shalev conjecture:

\begin{thm} \label{thm:bs}  Let $G$ be a finite simple group acting faithfully and transitively on a set $X$ of cardinality $n$.
With possibly finitely many exceptions, the proportion of derangements in $G$ is at least $.016$.
\end{thm}

The result fails for general transitive groups (indeed, there are easy examples where the proportion of derangements is exactly $1/n$).  It also fails for almost simple
groups  (it follows from our result that one can show the proportion of derangements for an almost simple group is at least
$\delta/\log n$).  Theorem \ref{thm:bs} was proved for alternating and  symmetric groups
by Luczak and Pyber \cite{LP}.  Indeed, they proved some stronger results.   Since the result is asymptotic, we can ignore sporadic groups
and so we consider finite simple groups of Lie type (as in \cite{LP}, we prove stronger results).

This is the fourth paper in a series beginning with \cite{FG4}, \cite{FG2}, \cite{FG1} and completes the proof.  Indeed, we prove the following (recall an element of a finite group of Lie type is regular semisimple if its centralizer in the corresponding algebraic group has connected component a (maximal) torus).

\begin{thm} \label{thm:regss}   There exists a $\delta > 0$ so that if  $G$ is a sufficiently
large finite simple group of Lie type acting faithfully and transitively on a set $X$,
then the proportion of elements which are both regular semisimple and derangements is at least $\delta$.
\end{thm}

Note that it is not always the case that there exist derangements which are semisimple.   The simplest example is to take
$G=PSL(2,5) \cong A_5$ acting on $5$ points.  The only derangements are elements of order $5$ which are unipotent.
On the other hand by \cite{GM}, aside from a very small number of simple finite groups of Lie type, there exist semisimple
regular conjugacy classes $C_1$ and $C_2$ such that no proper subgroup intersects both $C_i$ (whence in any action
either $C_1$ or $C_2$ consists of derangements); i.e. the group is invariably generated by $C_1$ and $C_2$ -- that is if $x_i \in C_i$, then
$G = \langle x_1, x_2 \rangle$.

Again, with possibly finitely many exceptions, one can take $\delta = .016$.    In \cite{FG4}, the result was proved for finite groups of Lie type
of bounded rank (and so in particular for the exceptional groups).  Another proof was given in \cite{FG2}.
Thus, it suffices to consider the finite classical groups  (e.g.,  linear, unitary, orthogonal and symplectic groups).  In \cite{FG2}, it was shown that aside
from the families (with regard to the natural module for the classical group):
\begin{enumerate}
\item  reducible subgroups;
\item  imprimitive subgroups (i.e. those stabilizing an additive decomposition of the spaces); and
\item extension field subgroups (i.e. those stabilizing an extension field structure on the natural module).
\end{enumerate}
that the proportion of derangements goes to $1$ as the rank goes to $\infty$.  Combining this result with \cite{FNP} gives
Theorem \ref{thm:regss} for these actions.

In \cite{FG1}, reducible subgroups were considered and Theorems \ref{thm:bs} and \ref{thm:regss} were proved in that case.
Moreover, it was shown that for the action
on an orbit of either totally singular or nondegenerate subspaces (of dimension at most $1/2$ the ambient space),
the proportion of derangements goes to $1$ if the dimension of
the subspaces tended to $\infty$.

In this paper, we deal with the last two families. We will prove:

\begin{thm} \label{thm:wreath}  There exist universal positive constants $A$ and $\delta$ satisfying the following.
Let $G$ be a finite classical group with natural module $V$ of dimension $n$, and
  $V = V_1 \oplus \ldots \oplus V_b$.   Let $H$ be the stabilizer of this decomposition and assume that
$H$ is irreducible on $V$.   Then the proportion of elements of $G$ contained in a conjugate of $H$ (for some $b$)
 is at most $A/n^{\delta}$.
\end{thm}

We also prove the following result for $GL$.

\begin{thm} \label{thm:PG}  There exist  positive constants $A$ and $\delta$ satisfying
the following.  Let $G=GL(n,q)$.    Let $X(G)$ denote the union  of all   irreducible subgroups of $G$
not containing $SL(n,q)$.   Then the proportion of elements in any given coset
of $SL(n,q)$ contained in $X(G)$ is at most $A/n^{\delta}$.
\end{thm}

In the previous result, $q$ can increase or be fixed.  This result was proved by Shalev \cite{Sh} for $GL(n,q)$ with $q$
fixed using deep work of Schmutz \cite{Sc} on orders of elements in general linear groups. Our technique is
more elementary.    In fact, one cannot do much better than the previous result.     Suppose that $n = 2m$.  The proportion of
elements contained in $GL(m,q) \times GL(m,q) < GL(m,q) \wr S_2 $  for $q$ large is approximately the same as
the proportion of elements in $S_n$ which fix a subset of size $m$.   By \cite{EFG}, this is approximately
of order $n^{-\delta} (1+\log \frac{n}{2})^{-3/2}$

$$
\delta =  1  - \frac{1 + \log \log 2} {\log 2}.
$$

The analog of Theorem \ref{thm:PG} is true for the other classical groups.  However, the proof requires some new results and will be proved
in a sequel (where we will also give an application to probabilistic generation).   We do show  the following:

\begin{thm} \label{thm:large}   Let $G$ be a finite classical group with natural module $V$ of dimension $n$.
Assume that $G$ is defined over $\mathbb{F}_q$.   Let $X(G)$ denote the union of all   irreducible subgroups of $G$
not containing the derived subgroup of $G$  (if $q$ is even and $G=Sp(2n,q)$, we exclude the subgroups $O^{\pm}(2n,q)$ from $X(G)$).
Let $Y(G)$ denote the set of regular semisimple elements contained in $X(G)$.
\begin{enumerate}
\item $ \lim_{n \rightarrow \infty} |Y(G)|/|G| = 0$;  and
\item  $ \lim_{\min \{n,q\} \rightarrow \infty}   |X(G)|/|G| = 0.$
\end{enumerate}
\end{thm}

Luczak and Pyber \cite{LP} proved the analog of Theorem \ref{thm:PG}  for symmetric and alternating
groups (with irreducible replaced by transitive).  Their result has been recently improved by Eberhart,
Ford and Green \cite{EFG}.

Section \ref{recall} recalls bounds from our paper \cite{FG2} on the number and sizes of conjugacy classes in
finite classical groups. It also recalls needed results from the paper \cite{FG1} on derangements in subspace actions of finite classical groups.
Section \ref{sec:sym} contains some results on Weyl groups that we require.

	Section \ref{wreathprod} proves a strengthening of the Boston-Shalev conjecture for stabilizers of imprimitive subgroups in the case of large rank. For example it shows that the proportion of elements of any coset of $SL(n,q)$ in $GL(n,q)$ which are contained in a conjugate of the wreath product $GL(m,q) \wr S_k$ goes to $0$ as $n=mk \rightarrow \infty$. Moreover it is shown that the same is true for families of maximal subgroups, i.e. that the proportion of elements of any coset of $SL(n,q)$ of $GL(n,q)$ which are contained in a conjugate of $GL(m,q) \wr S_k$ for some $m,k$ such that $mk=n$ (with $k>1$) also goes to $0$ as $n \rightarrow \infty$. This behavior is qualitatively different from the case of subspace actions. Namely it is proved in \cite{FG1} that as $k \rightarrow \infty$, the proportion of elements fixing a k-space tends to 0. But this does {\it not} hold for families; the probability that a random element fixes a $k$ space for some $1 \leq k \leq n/2$ tends to $1$ as  $n$ tends to infinity.
	
	Section \ref{extensionfield} proves the Boston-Shalev conjecture for extension field subgroups in the case of large rank. For example it shows that the proportion of elements in a coset $gSL(n,q)$ of $GL(n,q)$ which are both regular semisimple and contained in a conjugate of $GL(n/b,q^b).b$ is at most $A/n^{1/2}$, for a universal constant $A$. The $.b$ notation means semidirect product with the cyclic group of order b generated by the map $x \rightarrow x^q$ on $\F_{q^b}^*$.

    Section \ref{sec:proofs} shows how the five theorems stated in the introduction follow from earlier results (the proof of Theorem \ref{thm:PG}
    also uses the appendix). In the appendix we use generating functions to prove a strengthening of the Boston-Shalev conjecture for $GL$ in the case of extension field subgroups of large rank. Namely the appendix shows that the proportion of elements (not necessarily regular semisimple) in a coset $gSL(n,q)$ of $GL(n,q)$ which are contained in a conjugate of $GL(n/b,q^b).b$ goes to $0$ as $n \rightarrow \infty$. The argument does not easily generalize to the other classical groups. In a follow-up paper, we use a different method to prove this strengthening (and so the analog of Theorem \ref{thm:PG}) for the other classical groups.

\section{Background} \label{recall}

This brief section recalls some bounds from our papers \cite{FG1} and \cite{FG2}.
Let $k(G)$ denote the number of conjugacy classes of $G$.  More generally, if $N$
is a normal subgroup of $G$ and $g \in G$, let $k(Ng)$ denote the number of $N$-orbits
on the coset $Ng$.  By \cite{FG2}, $k(Ng)$ is precisely the number of $g$-stable conjugacy
classes in $N$.

First note that $k(GL(n,q)) \leq q^n$ and that $k(U(n,q)) \leq 8.26 q^n$ \cite{MR}.

From \cite{FG2}, we have upper bounds on the number of conjugacy classes in a finite classical group
of the form $cq^r$ where c is an explicit constant and $r$ is the rank  (indeed for the simply connected
groups, one gets bounds of the form $q^r + dq^{r-1}$ for an explicit $d$).

\centerline
{{\sc Table 1} \quad  Class Numbers for Classical Groups}~
\vspace{0.3cm}
\begin{center}
\begin{tabular}{|c||c|c|} \hline
 $G$ & $k(G)  \le $ &   Comments   \\ \skipa \hline \hline
$SL(n,q)$  & $ 2.5 q^{n-1}$ &  \\ \hline $SU(n,q)$  & $8.26 q^{n-1}$  &
\\  \hline $Sp(2n,q)$ &  $10.8q^n$ &   $q$ odd   \\ \hline
$Sp(2n,q)$  &  $15.2q^n$ &   $q$ even   \\ \hline
 $SO(2n+1,q)$ & $7.1q^n$    &   $q$ odd \\ \hline
 $\Omega(2n+1,q)$ & $7.3q^n$ &   $q$ odd \\ \hline
 $SO^{\pm}(2n,q)$ &  $7.5q^n$ &   $q$ odd \\ \hline
$\Omega^{\pm}(2n,q)$ &  $6.8q^n$         &   $ q$ odd \\ \hline
$O^{\pm}(2n,q)$ &  $9.5q^n$         &   $q$ odd \\ \hline
$SO^{\pm}(2n,q)$ &  $14q^n$         &   $q$ even \\ \hline
$O^{\pm}(2n,q)$ &  $15q^n$         &   $q$ even \\ \hline
\end{tabular}
\end{center}

Concerning centralizer sizes, the following lower bound is proved in \cite{FG2}.

\begin{theorem} \label{smallcent}
\begin{enumerate}

\item  Let $G$ be a connected
simple algebraic group of rank $r$  of adjoint type over a field of positive
characteristic. Let $F$ be a Steinberg-Lang endomorphism of $G$ with $G^F$
a finite Chevalley group over the field $\F_q$.   There is an absolute constant
$A$ such
that for all $x \in G^F$,
 $$
 |C_{G^F}(x)| >  \frac{q^{r}}{A (1 + \log_q r)}.
 $$

\item There is a universal constant $A$ such that for all $x \in GL(n,q)$,
\[ |C_{GL(n,q)}(x)| > \frac{q^n}{A (1 + \log_q n)}. \]

\item There is a universal constant $A$ such that for all $x \in U(n,q)$,
\[ |C_{U(n,q)}(x)| > \frac{q^n}{A (1 + \log_q n)}. \]

\end{enumerate}

\end{theorem}

Regarding derangements in subspace actions of finite classical groups, we recall the following results from
\cite{FG1}.

\begin{theorem} \label{GLsub} For $1 \leq k \leq n/2$, the proportion of elements of any coset of $SL(n,q)$ in $GL(n,q)$ which fix a $k$-space is at most $A/k^{.005}$, for $A$ a universal constant.
\end{theorem}

\begin{theorem} \label{Usub}
For $1 \leq k \leq n/2$, the proportion of elements of any coset of $SU(n,q)$ in $U(n,q)$ which fix a nondegenerate $k$-space is at most $A/k^{.005}$, for $A$ a universal constant, and the proportion of elements of any coset of $SU(n,q)$ in $U(n,q)$ which fix a totally singular $k$-space is at most $A/k^{.25}$, for $A$ a universal constant.
\end{theorem}

\begin{theorem} \label{Spsub}
For $1 \leq k \leq n/2$, the proportion of elements of $Sp(2n,q)$ which fix a nondegenerate $2k$ space is at most $A/k^{.005}$, for $A$ a universal constant, and for $1 \leq k \leq n$, the proportion of elements of $Sp(2n,q)$ which fix a totally singular $k$-space is at most $A/k^{.25}$, for $A$ a universal constant.
\end{theorem}

\begin{theorem} \label{Osub} For $1 \leq k \leq n/2$, the proportion of elements of $SO^{\pm}(n,q)$ which fix a nondegenerate $k$-space is at most $A/k^{.005}$ for $A$ a universal constant, and the proportion of elements of $SO^{\pm}(n,q)$ which fix a totally singular $k$-space is at most $A/k^{.25}$, for $A$ a universal constant.
\end{theorem}

\section{Some results on Weyl  groups} \label{sec:sym}

We record some results about Weyl groups that will be used in Sections \ref{wreathprod} and  \ref{extensionfield}.

For $x \in S_k$, we define $\orb(x)$ as the number of orbits of $x$ and $\ind(x) = k - \orb(x)$.  Note
that $\ind(x)$ is the also the minimal number $d$ such that $x$ is a product of $d$ transpositions.

\begin{lemma} \label{lem:index}  If $x \in S_k$ satisfies $\ind(x) < k/2$, then $x$ fixes subsets of every size from $1$ to $k$.
\end{lemma}

\begin{proof}  If $k \le 2$, then $x =1$ and the result is clear.   We may assume that $x \ne 1$.  Let $d$ be the length
of the largest cycle of $x$.  By induction, $x$ fixes subsets of every size at most $k-d$.  If $d \le k/2$, then
the result follows (since it is enough to check subsets of size up to $k/2$).   If $d > k/2$, then it follows that $x$ is a $d$-cycle  with
$d =  (k+1)/2$ (in particular, $k$ is odd).  Thus, $x$ has $(k-1)/2$ fixed points and again the result is clear.
\end{proof}

To see that Lemma \ref{lem:index} is sharp, note that a fixed point free involution $x \in S_k$ fixes no subsets of odd size and also satisfies $\ind(x)=k/2$.

\begin{lemma} \label{binomialbound} For $0 < t <1 $, and $r \geq 1$, the coefficient of $u^r$ in $(1-u)^{-t}$ is at most $te^tr^{t-1}$. \end{lemma}

\begin{proof} This coefficient is equal to $\frac{t}{r} \prod_{i=1}^{r-1} (1+\frac{t}{i})$. Taking logarithms base e, one sees that
\begin{eqnarray*}
\log \left[\prod_{i=1}^{r-1} \left( 1+\frac{t}{i} \right) \right] & = & \sum_{i=1}^{r-1} \log \left( 1+\frac{t}{i} \right)\\
& \leq & \sum_{i=1}^{r-1} \frac{t}{i}\\
& \leq & t (1+ \log(r-1)).
\end{eqnarray*} Taking exponentials one sees that the sought proportion is at most $te^tr^{t-1}$.
\end{proof}

\begin{lemma} \label{lem:bcycle}
For $b|n$, the proportion of elements in $S_n$ all of whose cycles have length
divisible by $b$ is at most $1.2/n^{1-1/b}$.
\end{lemma}

\begin{proof} By the cycle index of the symmetric groups (reviewed in the prequel \cite{FG1}), the sought proportion is the
 coefficient of $u^n$ in \[ \prod_{i \geq 1} e^{\frac{u^{ib}}{ib}} =
 (1-u^b)^{-1/b}.\] Now apply Lemma \ref{binomialbound} with $r=n/b$
 and $t=1/b$ to get an upper bound of
 \[ \frac{e^{1/b}}{b^{1/b} n^{1-1/b}} \leq 1.2/n^{1-1/b} .\]
\end{proof}

\begin{cor} \label{allcycle} The proportion of elements in $S_n$ with all cycle lengths divisible by some prime $b$ is at most $A/n^{1/2}$
for some universal constant $A$. \end{cor}

\begin{proof} It follows from Lemma \ref{lem:bcycle} that the sought proportion is at most
\[ 1.2 \sum_{b|n} \frac{1}{n^{1-1/b}}, \] where the sum is over prime divisors $b$ of $n$. Since $n$ has at most $\log_2(n)$ distinct
prime factors, this is at most
\[ \frac{1.2}{n^{1/2}} \left[ 1 + \frac{\log_2(n)}{n^{1/6}} \right], \] and the result follows. \end{proof}

We conclude this section with a result about signed permutations.

\begin{lemma}  \label{lem:type B Weyl}   Let $W = B_n$ be the Weyl group of type B.  The proportion of elements in $W$ such that all odd cycles
have a given type and all even cycles have a given type is at most $\frac{1}{\sqrt{\pi n}}$.
\end{lemma}

\begin{proof} We prove the lemma in the case that all even cycles have positive type, and all odd cycles have negative type; the other
cases are similar. As in \cite{FG1}, we apply the cycle index of the groups $B_n$. For a signed permutation $\pi$, let $n_i(\pi)$ denote the
number of positive $i$-cycles of $\pi$, and let $m_i(\pi)$ denote the number of negative $i$-cycles of $\pi$. The cycle index states that
\[ 1 + \sum_{n \geq 1} \frac{u^n}{2^n n!} \sum_{\pi \in B_n} \prod_i x_i^{n_i(\pi)} y_i^{m_i(\pi)} = \prod_{i \geq 1} e^{u^i(x_i+y_i)/(2i)}.\]
Setting \[ x_1=x_3=x_5 = \cdots = 0, \ \ y_2=y_4=y_6= \cdots = 0, \]
\[ x_2=x_4=x_6= \cdots = 1, \ \ y_1=y_3=y_5= \cdots = 1, \] gives that the proportion of elements of $W$ with all even cycles positive and
all odd cycles negative is the coefficient of $u^n$ in \[ \prod_{i \geq 1} e^{u^i/(2i)} = (1-u)^{-1/2} .\] Arguing as in \cite{FG1} shows
that this proportion is at most $\frac{1}{\sqrt{\pi n}}$.
\end{proof}

\section{Stabilizers of imprimitive subgroups} \label{wreathprod}

	This section proves a strengthening of the Boston-Shalev conjecture for wreath products when the
rank $n \rightarrow \infty$.

Subsection \ref{wreathprel} develops some preliminaries for the wreath product case. Subsections \ref{wreGL},
\ref{wreGU}, \ref{wreSp}, and \ref{wreO} treat wreath products for the general linear, unitary, symplectic,
and orthogonal groups respectively. Subsection \ref{isotropic} considers stabilizers for pairs of totally
isotropic subspaces.

\subsection{Preliminaries for wreath products} \label{wreathprel}

	Let $p(n)$ denote the number of partitions of an integer $n$; this is equal to the number of conjugacy classes of $S_n$. The series $p(n)$ begins with 1,2,3,5,7,11,15,22. Lemma \ref{Wallbound}, proven in the textbook \cite{VW}, gives a useful upper bound on $p(n)$.

\begin{lemma} \label{Wallbound} Let $p(n)$ be the number of partitions of $n$. For $n \geq 2$,
\[ p(n) \leq \frac{\pi}{\sqrt{6(n-1)}} e^{\pi \sqrt{2n/3}} .\]
\end{lemma}

{\it Remark:} Hardy and Ramanujan \cite{HR} give an elementary proof that $p(n) < \frac{K}{n} e^{2 \sqrt{2n}}$ for a universal constant $K$. In truth $p(n)$ is asymptotic to $\frac{1}{4n \sqrt{3}} e^{\pi \sqrt{2n/3}}$, as discussed in \cite{A}.

	Next we recall a description (proved in \cite {JK}) of conjugacy classes of the wreath product $G \wr S_k$, where $G$ is any finite group. The classes are parameterized by matrices where

\begin{enumerate}
\item The rows are indexed by conjugacy classes of $G$.
\item The number of columns is $k$.
\item Letting $a_{h,i}$ denote the $(h,i)$ entry of the matrix, the $a_{h,i}$ are nonnegative integers satisfying $\sum_{h,i \geq 1} i \cdot a_{h,i}=k$.
\end{enumerate}

	To determine the data corresponding to an element $(g_1,\cdots,g_k;\pi)$ in the wreath product, to each i-cycle of $\pi$ one associates a conjugacy class $C$ of $G$ by multiplying (in the order indicated by the cycle of $\pi$) the $g$'s whose subscripts form the entries of the cycle of $\pi$ and letting $C$ be the conjugacy class in $G$ of the resulting product. This contributes 1 to the entry of the matrix whose row entry is indexed by $C$ and whose column number is $i$.

\begin{lemma} \label{weylbcest} Suppose that $n > 1$.
\begin{enumerate}
\item  Let $W(B_n)$ be the Weyl group of type $B_n$.  Then
\[ k(W(B_n)) \le  \frac{(n+1)  \pi^2} {6(n-1)} e^{2\pi \sqrt{2n/3}}.\]
\item  Let $W(D_n)$ be the Weyl group of type $D_n$.  Then
\[ k(W(D_n)) \le  \frac{(n+1)  \pi^2} {3(n-1)} e^{2\pi \sqrt{2n/3}} .\]
\end{enumerate}
\end{lemma}

\begin{proof}   Since $[W(B_n):W(D_n)]=2$,  (2) follows from (1).

By the above description of conjugacy classes of wreath products, the conjugacy classes of $W(B_n)$ are indexed by
pairs of partitions $(\alpha, \beta)$ where $\alpha$ is a partition of $a$ and $\beta$ is a partition of $n-a$ for
$0 \le a \le n$.   It follows that $k(W(B_n)) \le (n+1) k(S_n)^2$ and so by Lemma \ref{Wallbound},
$$
k(W(B_n)) \le \frac{(n+1)  \pi^2} {6(n-1)} e^{2\pi \sqrt{2n/3}}.
$$
\end{proof}

Next we proceed to the main results of this section.

\subsection{GL(n,q)} \label{wreGL}

	To begin we consider cosets of $SL(n,q)$ in $GL(n,q)$, with $q$ fixed and $n \rightarrow \infty$. We note that Theorem \ref{Gl} was proved by Shalev \cite{Sh}
	for $q$ fixed by using deep work of Schmutz \cite{Sc} on the order of a random matrix. Our method is more elementary and extends to the other finite classical groups.
	
	\begin{theorem} \label{Gl}  There exist positive constants $B$ and $\delta$ so that the following holds (independently of
	$n, q, m$).   Let $n = mk$ and let
	$G=GL(n,q)$.  Set $H= GL(m,q) \wr S_k$.  The proportion of elements of $G$ in a coset of $SL(n,q)$ which are conjugate to an element
	of $H$ is less than $B/n^{\delta}$.
	\end{theorem}
	
	Recall \cite[Theorem 315]{HW} that $d(n)$, the number of divisors of $n$, decays faster than
         any power of $n$.  Thus, the conclusion of the theorem holds allowing all possible $m,k$.

         \begin{proof}  Let $H_0 = GL(m,q) \times \ldots \times GL(m,q)$, where there are $k$ copies of $GL(m,q)$.  Consider cosets of the form $xH_0$ where $x \in S_k$ and $\ind(x) \ge k/2$.
         Recall that $k(GL(m,q)) \le q^m$.
         By the description of conjugacy classes in a wreath product in the previous subsection, the number of $H_0$ orbits on the coset $xH_0$ is $k(GL(m,q))^{\orb(x)} \le q^{n/2}$.
         Thus the total number of $G$-conjugacy classes intersecting $xH_0$ for any such $x$ is at most
         \[ p(k) q^{n/2} \leq
          \frac{\pi}{\sqrt{6(k-1)}} e^{\pi \sqrt{2k/3}} q^{n/2}.
         \] Here $p(k)$ denotes the number of partitions of $k$ and we have used Lemma \ref{Wallbound}.

         Using the estimate for the smallest centralizer size (Lemma \ref{smallcent}) gives that the proportion of elements of $G$ conjugate to some such
         element $xH_0$ is at most (for a universal constant $A$)

         \[
          \frac{\pi}{\sqrt{6(k-1)}} e^{\pi \sqrt{2k/3}}A (1 + \log_q n)/ q^{n/2} < B/n^{\delta}.
         \] Even after multiplying by $q-1$, we still have the same estimate giving the result for each coset.

         Now consider cosets $xH_0$ where $\ind(x) < k/2$.  By Lemma \ref{lem:index}  $x$ fixes subsets of every possible size and so any element
         in $xH_0$ fixes a subspace of dimension $n/2$ (for $k$ even) and $(k-1)m/2$ (for $k$ odd).  In particular, every such element fixes a $d$
         dimensional space for some fixed $d \ge n/4$.
         By  Theorem \ref{GLsub}, it follows that the proportion of elements in a given coset of $SL(n,q)$
         which fix a subspace of that dimension is at most $C/(n/4)^{.005}$ and the result follows.
                  \end{proof}

\subsection{U(n,q)} \label{wreGU}

	Using the same method as in the general linear case, one establishes analogous results for the unitary groups.  The proof here requires
 a little extra effort since we only have the bound $k(U(m,q)) < 8.3 q^m$.
	
	\begin{theorem} \label{ShalevU}  There exist positive constants $B$ and $\delta$ so that the following holds (independently of
	$n, q, m$).   Let $n = mk$ and let
	$G=U(n,q)$.  Set $H= U(m,q) \wr S_k$.  The proportion of elements of $G$ in a coset of $SU(n,q)$ which are conjugate to an element
	of $H$ is less than $B/n^{\delta}$.   The same is true allowing all possible $m$.
	\end{theorem}
	
	\begin{proof} As for $GL$, the second statement follows from the first by \cite{HW}.    Set
	$H_0 = U(m,q) \times \ldots \times U(m,q)$. The proof is precisely along the lines of the $GL$ case.  The only
	difference is that we have to use the estimate $k(U(m,q)) < 8.3 q^m$.   We first consider cosets $xH_0$ where $x \in S_k$
	and $\ind(x) \ge k/2$.  The argument then gives the estimate that the proportion of elements of $G$ conjugate to an element in
	$xH_0$ for some such $x$ is at most
	$$
	\frac{\pi}{\sqrt{6(k-1)}} e^{\pi \sqrt{2k/3}} (8.3)^{k/2}  A (1 + \log_q n)/ q^{n/2} < B/n^{\delta}.
         $$
	
	This estimate is valid (even multiplying by $q+1)$ is valid as long as $q^m$ is bigger than $8.3$.
	  This holds unless $m=1$ or $m=2=q$   or $m=3$ and $q=2$.
          If $m=1$,
	we use the fact that $k(U(1,q))=q+1 \le (3/2)q$ and since $3/2 < 2 \le q$, the proof goes through.  If $m=2=q$, the same
	estimate holds (alternatively we note that if $m=2=q$, our subgroup $H$ is contained in $U(1,2) \wr S_n$).
Similarly,
	if $m=3$ and $q=2$, then $k(U(3,2))=24 < 3^3$ (and $H$ is again contained in  $U(1,2) \wr S_n$).
	Again, multiplying by $q+1$ shows the same estimate holds for each coset.
	
	Now consider cosets $xH_0$ where $x \in S_k$ and $\ind(x) < k/2$. Arguing as for $GL$ shows that any element in $xH_0$ fixes a nondegenerate subspace
	of fixed dimension at least $n/4$.
	Now apply Theorem \ref{Usub} to obtain the result.
	\end{proof}
		
\subsection{Sp(2n,q)} \label{wreSp}

	One has the following result (Theorem \ref{ShalevSp}), proven the same way as Theorems \ref{Gl} and \ref{ShalevU}. We use the bound $k(Sp(2m,q)) \le 15.2 q^m$. Note that $k(Sp(2,q)) \leq (3/2)q$ for $q$ even and $k(Sp(2,q)) \leq (7/3)q$ for $q$ odd. Similarly $k(Sp(4,q)) \leq (11/4)q^2$ for $q$ even and $k(Sp(4,q)) \leq (34/9)q^2$ for  $q$ odd.  We also have that $k(Sp(6,2))= 30 = (15/4)2^3$. We use these estimates, all of which follow from the generating functions for $k(Sp)$ in \cite{FG2}.

\begin{theorem} \label{ShalevSp}   There exist positive constants $B$ and $\delta$ so that the following holds (independently of
	$n, q, m$).   Let $n = mk$ and let
	$G=Sp(2n,q)$.  Set $H= Sp(2m,q) \wr S_k$.  The proportion of elements in  $G$  which are conjugate to an element
	of $H$ is less than $B/n^{\delta}$.   The same is true allowing all possible $m$.
	\end{theorem}
	
	\begin{proof} The last statement follows from the first by \cite{HW}.  Let $H_0$ be the subgroup $Sp(2m,q) \times \ldots \times Sp(2m,q)$.
	
	Arguing as for the unitary groups, we see that the proportion of elements of $G$ conjugate to an element in some coset $xH_0$
	with $x \in S_k$ and $\ind(x) \ge k/2$ is at most
	$$
	\frac{\pi}{\sqrt{6(k-1)}} e^{\pi \sqrt{2k/3}} (15.2)^{k/2}  A (1 + \log_q n)/ q^{n/2}.
	$$
	This is clearly at most $B/n^{\delta}$ unless possibly $m=1$  and $q < 16$ or $m=2$ and $q < 4$ or $m=3$ and $q=2$.
	Replacing the $15.2$ by the better estimates noted before the proof shows that the estimate still is valid in these cases.
	
	The bound for elements conjugate to a coset in $xH_0$ for $\ind(x) < k/2$ follows precisely as in the $GL$ or $U$ case.
	\end{proof}

\subsection{$\Omega(n,q)$} \label{wreO}

Finally, we consider orthogonal groups.   The proof is quite similar but there is one easy extra case to consider.
We prove the result for $SO^{\pm}(n,q)$ (which implies the
result for $\Omega$).

Let $X=O^{\pm}(n,q) = O(V)$ and $G =SO^{\pm}(n,q) = SO(V)$.
Assume that $n > 6$ and that if $q$ is even, then $n$ is even.
Write $V =V_1 \perp \ldots \perp V_k$, $k > 1$ where the $V_i$ are nondegenerate spaces of the same type.
Let $H$ denote the stabilizer of this decomposition in $X$.  Thus, $H \cong O(m,q) \wr S_k$
(the possibilities for the type of $V_i$ depend upon $k$ and the type of $V$).    Set $H_0=O(m,q) \times \ldots \times O(m,q)$.

We first consider the case that $m=1$ and $q$ is odd (if $q$ is even, then the stabilizer of an additive decomposition
of nondegenerate $1$-spaces is not irreducible and in particular not maximal).   Note that in this case
$H$ is the Weyl group of type $B_n$ and so is isomorphic to $\mathbb{Z}/2 \wr S_n$.  The intersection  of $H$
with $SO^{\pm}(n,q)$ will be the Weyl group of type $D_n$.
By Lemma \ref{weylbcest}, the number of conjugacy classes of $SO^{\pm}(n,q)$ that intersect $H$ is at most
$$ \frac{(n+1)  \pi^2} {3(n-1)} e^{2\pi \sqrt{2n/3}}.$$ Thus, by Theorem \ref{smallcent}, the proportion of
elements of  $SO^{\pm}(n,q)$ that intersect $H$ is at most
$$ A (1 + \log_q r)  \frac{(n+1)  \pi^2} {3(n-1)} e^{2\pi \sqrt{2n/3}}/q^r,$$
where $r = n/2$ if $n$ is even and $r =(n-1)/2$ for $n$ odd.     This is less than $B/n^{\delta}$ for a universal
$B$ and $\delta$.

If $m > 1$, the identical proof for the case of symplectic (or unitary) groups goes through.   Note that we have a slightly
better inequality for $k(O(m,q))$ than in the symplectic case.   Note that if $q$ is even, $m$ is also even
(because any odd dimensional space has a radical and so the stabilizer of such a decomposition is
not irreducible). The only additional wrinkle in the proof is that
we are working with $H \cap SO^{\pm}(n,q)$ and so when estimating the number of classes in a given coset
$xH$, we may have to multiply by $2$.   This is absorbed in the the constant for the size of the centralizer and so
causes no problems.

Thus, we have:

\begin{theorem} \label{ShalevO}  There exist positive constants $B$ and $\delta$ so that the following holds (independently of
	$n, q, m$).   Let $n = mk$ and let
	$G=SO^{\pm}(n,q)$.  Set $H= SO^{\pm}(n,q) \cap \left[O^{\pm}(m,q) \wr S_k \right]$.   If $q$ is even, assume that both $n$ and $m$ are even.
	For $n$ sufficiently large, the proportion of elements in  $G$  which are conjugate to an element
	of $H$ is less than $B/n^{\delta}$.   The same is true allowing all permissible $m$.
	\end{theorem}

\subsection{Stabilizers of pairs of totally isotropic subspaces} \label{isotropic}

We consider imprimitive subgroups permuting a direct sum decomposition of totally isotropic
spaces.  Note that if there are more than $2$, we would be contained in the stabilizer of an additive
decomposition of nondegenerate spaces (by taking pairs of the totally singular subspaces).
See \cite{As} for the description of the maximal subgroups of the classical groups.

To begin, we treat the case of $GL(n,q^2).2$ in $U(2n,q)$.

\begin{theorem}   There exist absolute positive constants $B$ and $\delta$ so
that for $n$ sufficiently large, the  proportion of elements of $U(2n,q)$ contained in a conjugate of $H:=GL(n,q^2).2$
is at most $B/n^{\delta}$.  \end{theorem}

\begin{proof} If an element of $U(2n,q)$ is contained in $GL(n,q^2)$, then it fixes a totally singular
$n$-dimensional space. By Theorem \ref{Usub}, the proportion of such elements tends to $0$ at the correct rate.

By Shintani descent (\cite{FG2}),  the   $H$-conjugacy classes in the nontrivial
coset of $GL(n,q^2)$ correspond exactly to conjugacy classes of $U(n,q)$. The number
of conjugacy classes of $U(n,q)$ is at most $8.26 q^n$.  In particular, the number of $U(2n,q)$-classes
in the nontrivial coset  is at most $8.26 q^n$.   Theorem \ref{smallcent} gives an  upper bound
for the size of a conjugacy class of $U(2n,q)$.  It follows that the proportion of elements of $U(2n,q)$
conjugate to an element of the nontrivial coset of $GL(n,q^2)$ is at most
\[ 8.26 q^n \frac{A (1+\log_q(2n))}{q^{2n}}, \] where $A$ is a universal constant. This is much smaller than
$B/n^{\delta}$.
 \end{proof}

A minor variant of the proof gives the following:

\begin{theorem} The proportion of elements of  in any coset of
$SU(2n,q)$ in $U(2n,q)$ contained in a conjugate of $H:=GL(n,q^2).2$
tends to $0$ as $n \rightarrow \infty$, uniformly in $q$. \end{theorem}

To treat the case of $GL(n,q).2$ inside of $Sp(2n,q)$ or $SO(2n,q)$, the following lemma will be helpful.
See also Lemma \ref{realU} for a related result about unitary groups.

\begin{lemma} \label{realGL} Let $G^+(n,q)$ denote the extension of $GL(n,q)$
generated by the inverse transpose involution $\tau$,
and let $k(GL(n,q)\tau)$ be the number of $G^+(n,q)$ conjugacy classes in the coset $GL(n,q) \tau$.
\begin{enumerate}
\item $k(GL(n,q)\tau)$ is the number of real conjugacy classes of $GL(n,q)$; and
\item $k(GL(n,q)\tau) \le 28 q^{\lfloor n/2 \rfloor}$;
\item The number of real conjugacy classes in $GL(n,q)$ is at most $28 q^{\lfloor n/2 \rfloor}$.
\end{enumerate}
\end{lemma}

\begin{proof}
By \cite{FG2}, the number of $G^+(n,q)$ classes in the outer coset is
precisely the number of $GL(n,q)$ classes
that are invariant under the involution.   Since any element of $GL(n,q)$ is conjugate
to its transpose, the $\tau$ invariant classes of $GL(n,q)$ are precisely  the real classes,
whence the first statement holds.
   The  second statement follows from the upper bound on $k(GL(n,q)\tau)$ in \cite{FG3}.
The final statement is now clear.
 \end{proof}

\begin{theorem}  There exist absolute positive constants $B$ and $\delta$
so that the following hold for
$n$ sufficiently large.
\begin{enumerate}
\item The proportion of elements of $Sp(2n,q)$ contained in a conjugate of
$GL(n,q).2$ is at most $B/n^{\delta}$.
\item The proportion of elements of $SO^{+}(2n,q)$ contained in a conjugate of
$GL(n,q).2$ is at most $B/n^{\delta}$.
\end{enumerate}
\end{theorem}

\begin{proof} As the argument is the same for both parts, we give details for the
first part. Note that if an element of
$Sp(2n,q)$ is conjugate to an element of $GL(n,q)$, then it fixes a totally singular
$n$-space. By Theorem \ref{Spsub},
the proportion of such elements is at most $B/n^{.25}$.

By Lemma \ref{realGL}, the number of $GL(n,q)$ classes in the outer coset in $GL(n,q).2$ is at most
$28 q^{\lfloor n/2 \rfloor}$. By Theorem \ref{smallcent}, any conjugacy class of
$Sp(2n,q)$ has size at most
\[ \frac{A (1+\log_q(n)) |Sp(2n,q)|}{q^n} , \]
which implies the result.
\end{proof}

\section{Extension field subgroups} \label{extensionfield}

	This section analyzes extension field subgroups in the case
where the rank approaches infinity. The standard extension
field cases for the groups $GL$, $U$, $Sp$, $O$ are treated in Subsections
\ref{GLsubsec}, \ref{GUsubsec}, \ref{Spsubsec}, and \ref{Osubsec} respectively.
Subsection \ref{special} considers some special cases, namely $U(n,q).2$ in
$Sp(2n,q)$ and $U(n,q).2$ in $SO^{\pm}(2n,q)$.

First we consider an example which shows that the bounds cannot be improved too much.

\begin{example}   Let $n = 2m$ be a positive integer.  Set  $G=GL(n,q)$ and consider $H=GL(m,q^2)< G$.  Fix $m$ and let $q$ grow.
Then almost all elements of $H$ are regular semisimple (even in $G$).   Let $X$ be the set of elements
in $S_n$ in which all cycles have even length.    We see that $H$ contains
conjugates of any maximal tori $T_w$ of $G$  where $w \in X$.
It follows that
$$
\lim_{q \rightarrow \infty}  |\cup_{g \in G, w \in X}   gT_wg^{-1}|  = |X|/n!.
$$ From \cite{FG1}, $|X|/n!$ is asymptotic to $\sqrt{\frac{2}{\pi n}}$.
Thus,  the proportion of derangements in the action of $G$ on $N_G(H)$ is of order $n^{-1/2}$ for $q$ large.
\end{example}

One can construct a similar example with $n$ increasing by letting $q$ increase very rapidly.

\subsection{GL(n,q)} \label{GLsubsec}

	Recall that $GL(n/b,q^b).b$ denotes the semidirect product of $GL(n/b,q^b)$ with the cyclic group of order $b$ generated by the map $x \rightarrow x^q$ on $\F_{q^b}^*$; this group is maximal when $b$ is prime \cite{As}. The main result of this subsection is that the proportion of elements in a given coset of $SL(n,q)$
in $GL(n,q)$ which are both regular semisimple and contained in a conjugate of $GL(n/b,q^b).b$ is at most $A/n^{1/2}$ for a universal constant $A$. A stronger result (requiring a more intricate proof) is in the appendix.

\begin{theorem} \label{Slconjout} Let $b$ be prime. The number of $GL(n,q)$ classes of the group
$GL(n/b,q^b).b$ outside $GL(n/b,q^b)$ is at most $(b-1)q^{n/b} \leq 2q^{n/2}$. \end{theorem}

\begin{proof}
Let $H$ denote $GL(n/b,q^b).b$. Fix a generator of the subgroup of order b, say x with x inducing
the q-Frobenius map on $H_0=GL(n/b,q^b)$.

Fix $0 < i < b$. By Shintani descent,  there is a bijection between $H$-conjugacy classes in the coset
$H_0x^i$ and conjugacy classes in $GL(n/b,q)$.   So there are at most $(b-1)q^{n/b}$
conjugacy classes in $H \setminus{H_0}$. This is easily seen to be at most $2q^{n/2}$.
\end{proof}

This result is sufficient to prove the Boston--Shalev result for $SL(n,q)$.

\begin{cor} \label{SLBSEX}  Let $n \ge 3$.  Let $b$ be a prime dividing $n$.  There is a universal constant $B$ such
that the
 proportion of elements in $G:=SL(n,q)$ which are contained
in a conjugate of $H:=GL(n/b,q^b).b$ is at most
$$\frac{1}{b}  +\frac{ B(1 + \log_q n)}{q^{n/2-1}}.$$
\end{cor}

\begin{proof}  By the previous result, there are at most
$2q^{n/2}$  conjugacy classes of $GL(n,q)$ that intersect $H \setminus {H_0}$
where $H_0 = GL(n/b,q^b)$.  By  Theorem \ref{smallcent}, this implies that the proportion  of elements of
$GL(n,q)$ which intersect some conjugate of $H \setminus{H_0}$ is at most
$$
\frac{B(1 + \log_q n)}{q^{n/2}}
$$
for some universal constant $B$.   Thus, the proportion of elements of $G$ contained in a conjugate of of $H \setminus{H_0}$
is at most $B'(1 + \log_q n)/q^{n/2-1}$.

Since $[N_G(H_0):H_0]  = b$, it follows that the union of the conjugates of $H_0$ contains at most $|G|/b$ elements,
whence the result.
\end{proof}

We now get some better estimates at least for $q$ large.

\begin{lemma} \label{option} Let $b$ be a prime divisor of $n$. If an element of $GL(n,q)$ is contained in a conjugate of $GL(n/b,q^b)$, then every
irreducible factor of its characteristic polynomial either has degree divisible by $b$ or has every Jordan block size occur
with multiplicity a multiple of $b$. \end{lemma}

\begin{proof} Consider an element $A \in GL(n/b,q^b)$.
We can write $A$ as a block diagonal matrix where the block diagonals are of the form $C_{p_i} \otimes J_i$  where
$C_{p_i}$ is the companion matrix of the irreducible polynomial $p_i \in \FF_{q^b}[x]$ and $J_i$ is a regular
unipotent matrix of the appropriate size.

Since $b$ is prime, there are two possibilities for $p_i$.  The first is that $p_i$ is defined over $\FF_q$ (and is of
course irreducible over $\FF_q$). The second is that $p_i$ has $b$ distinct Galois conjugates over $\FF_q$
so the product of these conjugates $f_i$ is defined and irreducible over $\FF_q$.
This proves the lemma.
\end{proof}

Now we proceed to the main result of this subsection.

\begin{theorem} \label{regssGL}   Let $b$ be a prime dividing $n$ with $n \ge 3$.
\begin{enumerate}
\item The proportion of elements in a coset $gSL(n,q)$ in $GL(n,q)$ which are
regular semisimple and contained in a conjugate of $GL(n/b,q^b).b$ is at most
$A/n^{1/2}$ for a universal constant $A$.

\item The proportion of elements in a coset $gSL(n,q)$ in $GL(n,q)$ which are regular semisimple and contained in a
conjugate of $GL(n/b,q^b).b$ for some prime $b$ is at most $A /n^{1/2}$ for a universal constant $A$.
\end{enumerate}
\end{theorem}

\begin{proof}    We argue exactly as in the proof of Corollary \ref{SLBSEX} to see that the proportion of elements in
a coset $gSL(n,q)$ contained in a conjugate of $GL(n/b,q^b).b \setminus{GL(n/b,q^b)}$ is at most
$$
\frac{B(1 + \log_q n)}{q^{n/2-1}}
$$
for some universal constant $B$.  Summing over all possible $b$ just multiplies the upper bound by
at most $\log_2(n)$ (since there are most $\log_2(n)$ possibilities for $b$).   This is still much
less than  $ A/n^{1/2} $.   Note we are not restricting to semisimple regular elements in this case.

Now we consider semisimple regular elements in a coset $gSL(n,q)$ contained in some conjugate of $GL(n,q^b)$.
By Lemma \ref{option}, any such element has characteristic polynomial a product of polynomials with all irreducible
factors having degree a multiple of $b$.   Thus, any such element is contained in a maximal torus $T_w$ with
$w \in S_n$ and all cycles lengths of $w$ being a multiple of $b$.    By Corollary \ref{allcycle}, the proportion of
$w \in S_n$ with that property is at most $A/n^{1/2}$ for a universal constant $A$.    Arguing as in \cite[\S5]{FG1},
this implies that the proportion of regular semisimple elements in $gSL(n,q)$ with this property is also at most
$A/n^{1/2}$.

Combining these two estimates proves (2) (and also (1)).

\end{proof}

\subsection{U(n,q)} \label{GUsubsec}

To begin we have the following unitary analog of Theorem \ref{Slconjout}.

\begin{theorem} \label{Uconjout} Let $b$ be an odd prime.  Then the  number of $U(n,q)$ classes in
$U(n/b,q^b).b$ outside $U(n/b,q^b)$ is at most $(b-1)k(U(n/b,q))$.  \end{theorem}

\begin{proof}   Set $H= U(n/b,q^b).b$ and $H_0 = U(n/b,q^b)$.   By Shintani descent,
the number of $H$-conjugacy classes
in any nontrivial coset of $H_0$ is $k(U(n/b,q))$, whence the result.
\end{proof}

The Boston--Shalev result follows in this case arguing precisely as in Corollary \ref{SLBSEX}.  It also follows
from our results below.

The following theorem is the main result of this subsection.

\begin{theorem}      Let $b$ be an odd prime dividing $n$.
\begin{enumerate}
\item The proportion of elements in a coset $gSU(n,q)$ in $U(n,q)$ which are both
regular semisimple and contained in a conjugate of $U(n/b,q^b).b$ is at most $A/n^{1/2}$ for a universal constant $A$.

\item The proportion of elements in a coset $gSU(n,q)$ in $U(n,q)$ which are both regular semisimple
and contained in a conjugate of $U(n/b,q^b).b$ for some odd prime $b$ is at most $A/n^{1/2}$ for a universal constant $A$.
\end{enumerate}
\end{theorem}

\begin{proof}   The proportion of elements in $gSU(n,q)$ contained in a conjugate of $U(n/b,q^b).b \setminus U(n/b,q^b)$
is at most
$$ (b-1)k(U(n/b,q)) B (1 + \log_q n)(q+1)/q^{n}, $$
for some universal constant (by Theorem \ref{Uconjout} and Theorem \ref{smallcent}).
Since $k(U(n/b,q)) \le 8.3 q^{n/b}$, the result holds for these elements.

Next consider the proportion of regular semisimple elements  in $gSU(n,q)$ contained in a conjugate of $U(n/b,q^b)$.
Any regular semisimple element of $U(n,q)$ contained in a conjugate of $U(n/b,q^b)$ is a contained in a maximal torus
of the latter group.   These are also maximal tori of the larger group and correspond to $T_w$ with $w \in S_n$ the Weyl
group of $U(n,q)$, where all cycles of $w$ have length divisible by $b$ (by precisely the same argument as for $GL$).

By \cite[\S5]{FG1} and Corollary \ref{allcycle}, the proportion of such elements (summing over all b) is at most
$B'/n^{1/2}$ for some absolute constant $B'$.   Thus (2)  and so also (1) hold.
\end{proof}

\subsection{Sp(2n,q)} \label{Spsubsec}

	This subsection analyzes the case of the symplectic groups.

\begin{thm} \label{thm:sp-ss}   Let $b$ be a prime dividing $n$.
\begin{enumerate}
\item The proportion of elements in $Sp(2n,q)$ which are regular
semisimple and  contained in a
conjugate of $Sp(2n/b,q^b).b$ is at most $\frac{A}{n^{1/2}}$ where $A$ is a universal constant.
\item The proportion of elements in $Sp(2n,q)$  which are regular semisimple
and contained in a conjugate of $Sp(2n/b,q^b).b$ for some
prime $b|n$ is at most $\frac{A} {n^{1/2}}$ where $A$ is a universal constant.
\end{enumerate}
\end{thm}

\begin{proof}  We will prove (2).  Then (1) follows immediately.

As usual, by Shintani descent, the number of $Sp(2n/b,q^b).b$ classes in an outer class is at most
$(b-1)k(Sp(2n/b,q)) < 15.2 (b-1) q^{n/b}$.    By Theorem \ref{smallcent},  the estimate
easily holds for such elements (summing over all prime divisors $b$ of $n$).

Now consider the conjugates of $Sp(2n/b,q^b)$.   By the argument for the $GL$ case, we see
that every factor of the characteristic polynomial of a regular semisimple element $g$ in $Sp(2n/b,q^b)$
has degree divisible by $b$.  Moreover, the centralizer of $g$ is contained in $Sp(2n/b,q^b)$.
Thus any element in $Sp(2n/b,q^b)$ is contained in a  conjugate of a maximal torus $T_w$ where
$w$ is in the Weyl group and has all cycles of length divisible by $b$.   By Corollary
\ref{allcycle}
and \cite[\S5]{FG1},  it follows that the proportion of regular semisimple elements conjugate to
an element of $Sp(2n/b,q^b)$ is at most $D/n^{1/2}$ for some constant $D$.   This gives the result.
\end{proof}

\subsection{O(n,q)} \label{Osubsec}

The proof for $SO$ is essentially identical to that of $Sp$.  The only difference in the argument
is to use strongly regular semisimple elements (i.e. semisimple elements whose characteristic
polynomials have distinct roots).    Note that if $b$ is odd, then the two orthogonal groups will
have the same type.   If $b=2$, then the big group must have $+$ type.

\begin{thm} \label{thm:so-ss}
\begin{enumerate}
\item For a prime number $b|n$, the proportion of elements in $SO^{\pm}(2n,q)$ which are both
strongly regular
semisimple and  contained in a
conjugate of $SO^{\pm}(2n/b,q^b).b$ is at most $\frac{A}{n^{1/2}}$ where $A$ is a universal constant.
\item The proportion of elements in $SO^{\pm}(2n,q)$  which are both strongly semisimple
regular and contained in a conjugate of $SO^{\pm}(2n/b,q^b).b$ for some
prime $b|n$ is at most $\frac{A}{n^{1/2}}$ where $A$ is a universal constant.
\end{enumerate}
\end{thm}

\subsection{Some special cases} \label{special}

In this subsection, we treat some special cases of extension field groups.

To begin we treat the case of $U(n,q).2$ contained in $Sp(2n,q)$
(recall that $U(n,q)$ is contained in $GL(n,q^2)$ and imbeds in $Sp(2n,q)$
via the embedding of $GL(n,q^2)$ in $GL(2n,q)$).

Lemma \ref{realU} upper bounds the number of real conjugacy classes of $U(n,q)$.

\begin{lemma} \label{realU} The number of real conjugacy classes of
$U(n,q)$ is equal to the number of real conjugacy classes of $GL(n,q)$
and is at most $28 q^{\lfloor n/2 \rfloor}$.
\end{lemma}

\begin{proof}   Let $C$ be a conjugacy class of $U(n,q)$.  Let $\bar{C}$ denote the corresponding conjugacy class in
the algebraic group $GL(n,\bar{\F_q})$. Note that since all centralizers in $GL(n,\bar{\F_q})$ are connected,
${\bar C} \cap U(n,q)$ and ${\bar C} \cap GL(n,q)$ are single  conjugacy classes in the corresponding finite group  (by Lang's theorem).

Note that if $C=C^{-1}$,  then $C$ is invariant
under the $q$-Frobenius map and so has a representative in $GL(n,q)$ and conversely.  Thus, the map
$C \rightarrow \bar{C} \cap GL(n,q)$ gives a bijection between real classes of $GL(n,q)$ and $U(n,q)$.
The result now follows by Lemma \ref{realGL}.
 \end{proof}

 \begin{thm} \label{them:U in Sp regss}  The proportion of elements of $Sp(2n,q)$ that are regular semisimple
 and conjugate to an element of $U(n,q).2$
is at most $\frac{A}{n^{1/2}}$ for a universal constant $A$.
\end{thm}

\begin{proof}  First consider classes of $Sp(2n,q)$ in the nontrivial coset of $U(n,q)$.
The number of $U(n,q)$ orbits on this coset is precisely the number of conjugacy classes
of $U(n,q)$ which are left invariant by the outer automorphism.   It is a straightforward exercise
to see that all such classes are real in $U(n,q)$.
Thus,  by Lemma \ref{realU},
the number of them is at most $28 q^{\lfloor n/2 \rfloor}$. Using Theorem \ref{smallcent} to upper bound the size of
a conjugacy classes of $Sp(2n,q)$, it follows that the proportion of elements of $Sp(2n,q)$ conjugate to an
element in the non-trivial coset of $U(n,q)$ is at most
\[ \frac{q^{n/2} C(1+\log_q(n))}{q^n}, \] for a universal constant $C$, whence the result holds for such elements.

If $g \in Sp(2n,q)$ is a regular semsimple element conjugate to an element of $U(n,q)$,  then $g$ is certainly
regular semisimple in $U(n,q)$ and so is contained in some maximal torus $T$ of $U(n,q)$.   Since
$U(n,q)$ and $Sp(2n,q)$ are both rank $n$ groups, $T$ is also a maximal torus of $Sp(2n,q)$.  By considering
the embedding of the maximal torus, we see that $T$ is conjugate to a maximal torus $T_w$ where $w$ is the
Weyl group (of type B) and all odd cycles have $-$ type and all even cycles have $+$ type.   By Lemma \ref{lem:type B Weyl},
and \cite[\S5]{FG1},  it follows that the proportion of elements
which are both regular semisimple and conjugate to an element of
$U(n,q)$ is at most $C'/n^{1/2}$ for a universal constant $C'$.  The result follows.
\end{proof}

The identical proof works for $SO$ (noting that the Weyl group of type D is a subgroup of index 2 in the
Weyl group of type B) and using strongly regular semisimple elements rather than semisimple regular elements.

 \begin{thm} \label{thm:U in O regss}  The proportion of elements of $SO^{\pm}(2n,q)$ that are strongly regular semisimple
 and conjugate to an element of $U(n,q).2$
is at most $\frac{A}{n^{1/2}}$ for a universal constant $A$.
\end{thm}

\section{Proofs of the theorems} \label{sec:proofs}

Theorems \ref{thm:wreath} and \ref{thm:large}  follow immediately by the results of
Sections \ref{wreathprod} and \ref{extensionfield}. Theorem \ref{thm:PG} also uses the appendix.

Note that the results of Sections \ref{wreathprod} and \ref{extensionfield} show that for the maximal imprimitive groups and
the extension field groups, the proportion of elements which are both regular semisimple (or strongly regular semisimple
for the orthogonal groups) and are not derangements (for at least one of the actions) goes to $0$ with $n$.  Since
the proportion of regular semisimple elements (or strongly regular semisimple elements) is greater than $.016$
\cite{FNP}, the proportion
of regular semisimple elements which are derangements in all such actions is at least $.016$ (for $n$ sufficiently large).
Combining this result with the main results of \cite{FG1, FG2, FG4} yields Theorem \ref{thm:regss}.
Theorem \ref{thm:regss} and the results of \cite{LP} on symmetric groups  yield Theorem \ref{thm:bs}.

\section*{Appendix}

The purpose of this appendix is to show that the proportion of elements of any coset $g SL(n,q)$ contained in a conjugate
of $GL(n/b,q^b).b$ for some prime $b$ is at most $A \cdot \log_2(n)/n^{1/4}$ for a universal constant $A$.
This strengthens Theorem \ref{regssGL} (which only considered regular semisimple elements). However the proof
technique does not easily extend to the other classical groups.

Let $N(q;d)$ denote the number of monic irreducible degree $d$ polynomials
over $\mathbb{F}_q$ with non-zero constant term.

\begin{lemma} \label{set1incycle}
\[ \prod_{d \geq 1} \prod_{i
\geq 1} \left(1-\frac{u^d}{q^{id}} \right)^{-N(q;d)} = (1-u)^{-1} .\]
\end{lemma}

\begin{proof} By switching the order of the infinite products, the lemma
follows from the well-known equation (see for instance \cite{F})
\[ \prod_{d \geq 1} (1-u^d)^{-N(q;d)} = \frac{1-u}{1-qu}.\]
\end{proof}

	Lemma \ref{compareN} will be helpful in upper bounding the proportion of elements of $GL(n,q)$ conjugate to an element of $GL(n/b,q^b)$.

\begin{lemma} \label{compareN} $N(q;db) \leq \frac{1}{b} N(q^b;d)$.
\end{lemma}

\begin{proof} Recall the Galois theoretic interpretation of roots of an irreducible polynomial as orbits under the Frobenius map. The left hand side is $1/(db)$ multiplied by the number of elements of $\mathbb{F}_{q^{db}}^*$ which form an orbit of size $db$ under the Frobenius map $x \rightarrow x^q$. The quantity $N(q^b;d)/b$ is $1/(db)$ multiplied by the number of elements of $\mathbb{F}_{q^{db}}^*$ which form an orbit of size $d$ under the map $x \rightarrow x^{q^b}$. The lemma follows. \end{proof}

	For $f(u)=\sum_{n \geq 0} f_n u^n$, $g(u)=\sum_{n \geq 0} g_n u^n$, we let the notation $f<<g$ mean that $|f_n| \leq |g_n|$ for all $n$. In the proof of Theorem \ref{countindot}, it will also be useful to have some notation about partitions. Let
$\lambda$ be a partition of some non-negative integer $|\lambda|$ into
parts $\lambda_1 \geq \lambda_2 \geq \cdots$. Let $m_i(\lambda)$ be the
number of parts of $\lambda$ of size $i$, and let $\lambda'$ be the
partition dual to $\lambda$ in the sense that $\lambda_i' = m_i(\lambda) +
m_{i+1}(\lambda) + \cdots$. It is also useful to define the diagram
associated to $\lambda$ by placing $\lambda_i$ boxes in the $i$th row. We
use the convention that the row index $i$ increases as one goes downward.
So the diagram of the partition $(5441)$ is

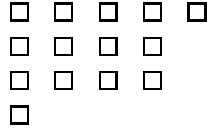
\[ \begin{array}{c c c c c} \framebox{} & \framebox{} & \framebox{}
& \framebox{} &  \framebox{} \\ \framebox{} &  \framebox{}& \framebox{} &
\framebox{} & \\
\framebox{} & \framebox{} & \framebox{} & \framebox{}&    \\ \framebox{} &
& & &
\end{array}
\] and $\lambda_i'$ can be interpreted as the size of the $i$th column.
The notation $(u)_m$ will denote $(1-u)(1-u/q) \cdots (1-u/q^{m-1})$.

\begin{theorem} \label{countindot} For $b$ prime, the proportion of elements in $GL(n,q)$ conjugate to an element of $GL(n/b,q^b) $ is at most $\frac{A}{n^{1/2}}$ where $A$ is a universal constant. \end{theorem}

\begin{proof} By Lemma \ref{option}, if an element of $GL(n,q)$ is contained in a conjugate of $GL(n/b,q^b)$, every irreducible factor of its characteristic
 polynomial either has degree divisible by $b$ or has every Jordan block size occur with multiplicity a multiple of $b$.
 By the cycle index of $GL(n,q)$ (see \cite{F} or \cite{St} for background), the proportion of such
 elements is at the most the coefficient of $u^{n/b}$ in \[ \prod_{d \geq 1} \prod_{i \geq 1} (1-u^d/q^{idb})^{-N(q;db)} \prod_{d \geq 1} \left[ \sum_{\lambda \in P_b} \frac{u^{|\lambda|d/b}}{c(q^d,\lambda) } \right]^{N(q;d)} \] where \[ c(q,\lambda)= \frac{1} {q^{\sum_i (\lambda_i')^2} \prod_i (1/q)_{m_i(\lambda)}} \] and $P_b$ is the set of partitions in which each part size occurs with multiplicity a multiple of $b$.

	By Lemma \ref{compareN} and Lemma \ref{set1incycle}, \[ \prod_{d \geq 1} \prod_{i \geq
1} (1-u^d/q^{idb})^{-N(q;db)} << \prod_{d \geq 1} \prod_{i
\geq 1} (1-u^d/q^{idb})^{-\frac{1}{b} N(q^b;d)} =
(1-u)^{-1/b}.\] By Lemma \ref{binomialbound}, the coefficient of $u^{r}$ in this
expression is at most $\frac{A}{b \sqrt{r}}$ where $A$ is a universal
constant.

	Next, we claim that the coefficient of $u^s$ in \[\prod_{d
\geq 1} \left[ \sum_{\lambda \in P_b}
\frac{u^{|\lambda|d/b}}{c(q^d,\lambda)} \right]^{N(q;d)}\] is at most $q^s$
divided by the minimum centralizer size of an element of
$GL(sb,q)$. To see this, observe that after expanding out the product,
the terms correspond to conjugacy classes of $GL(sb,q)$ with the
property that every Jordan block corresponding to an irreducible polynomial
occurs with multiplicity a multiple of $b$. These correspond to classes of $GL(s,q)$
and the number of them is at most $q^s$ \cite{MR}. To complete the proof of the claim, recall
from \cite{M} that \[ \prod_{\phi} c(q^d,\lambda_{\phi}) \] is the centralizer size of an element with
conjugacy data $\{\lambda_{\phi}\}$. Thus Theorem
\ref{smallcent} implies that the coefficient of $u^s$ in
\[\prod_{d \geq 1} \left[ \sum_{\lambda \in P_b}
\frac{u^{|\lambda|d/b}}{c(q^d,\lambda)} \right]^{N(q;d)}\] is at most
\[\frac{ A (1+ \log_q(bs))}{q^{(b-1)s}} \] for a universal constant $A$.

Thus the coefficient of $u^{n/b}$ in \[ \prod_{d \geq 1} \prod_{i \geq 1} (1-u^d/q^{idb})^{-N(q;db)} \prod_{d \geq 1} \left[ \sum_{\lambda \in P_b} \frac{u^{|\lambda|d/b}}{c(q^d,\lambda) } \right]^{N(q;d)} \] is at most

\begin{eqnarray*}
& & Coef. u^{n/b} in (1-u)^{-1/b} + Coef. u^{n/b} in \prod_{d \geq 1} \left[ \sum_{\lambda \in P_b} \frac{u^{|\lambda|d/b}}{c(q^d,\lambda) } \right]^{N(q;d)} \\
& & + \sum_{r=1}^{\frac{n}{b}-1} Coef. u^{r} in (1-u)^{-1/b} \cdot Coef. u^{n/b-r} in \prod_{d \geq 1} \left[ \sum_{\lambda \in P_b} \frac{u^{|\lambda|d/b}}{c(q^d,\lambda) } \right]^{N(q;d)} \\
& \leq & \frac{A}{b \sqrt{n/b}} + \frac{A(1+\log_q(n))}{q^{n-n/b}} + \sum_{r=1}^{n/b-1} \frac{A}{b \sqrt{r}}
\frac{(1+\log_q(n-br))}{q^{(b-1)(n/b-r)}}. \end{eqnarray*} Splitting the sum into two sums (one with $1 \leq r \leq n/(2b)$ and the other with $n/(2b) \leq r \leq n/b-1$) proves the theorem.	
\end{proof}

Now we prove the main results of this appendix.

\begin{theorem} \label{GLext}
\begin{enumerate}
\item For $b$ prime, the proportion of elements of $GL(n,q)$ contained in a conjugate of $GL(n/b,q^b).b$ is at most $\frac{A}{n^{1/2}}$ for a universal constant $A$.
\item The proportion of elements of $GL(n,q)$ contained in a conjugate of $GL(n/b,q^b).b$ for some prime $b$ is at most $\frac{A \cdot \log_2(n)}{n^{1/2}}$ for a universal constant $A$.
\end{enumerate}
\end{theorem}

\begin{proof} The second part of the theorem follows from the first part together with the fact that an integer $n$ has at most $\log_2(n)$ prime divisors; hence we prove part one.

By Theorem \ref{countindot}, the proportion of elements of $GL(n,q)$ conjugate to an element of $GL(n/b,q^b)$ is at most $A/n^{1/2}$ for a universal constant $A$. Now the number of elements of $GL(n,q)$ conjugate to an element of the group $GL(n/b,q^b).b$ but not to anything in $GL(n/b,q^b)$ is at most the number of conjugacy classes of $GL(n/b,q^b).b$ outside of $GL(n/b,q^b)$ multiplied by the maximum size of a $GL(n,q)$ class. These two quantities are bounded in
Theorems \ref{Slconjout} and \ref{smallcent} respectively. One concludes that the proportion of elements of $GL(n,q)$ conjugate to an element of $GL(n/b,q^b).b$ outside of $GL(n/b,q^b)$ is at most
\[ \frac{A q^{n/2} (1+\log_q(n))}{q^n} \leq A/n^{1/2} .\]
\end{proof}

Let us consider the same problem for $SL(n,q)$ or more generally for a fixed coset of $SL(n,q)$.
Since $[GL(n,q):SL(n,q)] = q-1$, the previous result implies that the proportion of elements in a given coset of
$SL(n,q)$ conjugate to an element of $GL(n/b,q^b).b$ is at most $(q-1)A/n^{1/2}$ for a universal constant $A$.
So if $q < n^{1/4}$, we see that the proportion of elements of $gSL(n,q)$ in a conjugate of $GL(n/b,q^b).b$ is at most
$A/n^{1/4}$.

Suppose that $q \ge n^{1/4}$.    Then the proportion of elements in $gSL(n,q)$ which are not regular
semisimple is at most $C/q  \le C/n^{1/4}$ for a universal constant $C$.
Arguing as above,  we
see that every regular semisimple element in $GL(n,q)$ contained in $GL(n/b,q^b)$
has all irreducible factors of its characteristic polynomial of degree a multiple of $b$.    Moreover, we see that the centralizer of such
an element (a maximal torus) in $GL(n,q)$ is contained in $GL(n/b,q^b)$.    So a maximal torus $T_w$ is conjugate to a subgroup of
$GL(n/b,q^b)$ if and only if all cycles of $w$ have length divisible by $b$.  By Lemma \ref{lem:bcycle},  the proportion of elements in $S_n$
with this property is at most
$A/n^{1-1/b}$ for some universal constant $A$. By \cite[\S5]{FG1},  this implies that the proportion of elements which are regular
semisimple and contained in a conjugate of $GL(n/b,q^b)$ in any fixed coset of $SL(n,q)$  is at most $A/n^{1-1/b}$.
Arguing as in the previous theorem shows that the proportion of elements conjugate to an element of $GL(n/b,q^b).b$ outside
of $GL(n/b,q^b)$ is at most $A/n^{1/2}$.
Summarizing, if $q \ge n^{1/4}$, we have that the proportion of elements
of $gSL(n,q)$ which are contained in some conjugate of $GL(n/b,q^b).b$ is at most $D/n^{1/4}$.   So we have proved the following:

\begin{theorem} \label{SLext}
\begin{enumerate}
\item For $b$ prime, the proportion of elements of any coset $gSL(n,q)$ contained in a conjugate of
$GL(n/b,q^b).b$ is at most $\frac{A}{n^{1/4}}$ for a universal constant $A$.
\item The proportion of elements of any coset $gSL(n,q)$ contained in a conjugate of $GL(n/b,q^b).b$ for some prime $b$
is at most $\frac{A \cdot \log_2(n)}{n^{1/4}}$ for a universal constant $A$.
\end{enumerate}
\end{theorem}

Almost certainly,  the $n^{1/4}$ can be replaced by $n^{1/2}$ and the $\log$ factor in (2) can be removed.

\end{document}